\documentclass[12pt]{article}
\usepackage[utf8]{inputenc}
\usepackage{amsmath}
\usepackage{amssymb}
\usepackage{mathrsfs}
\usepackage{amsthm}
\usepackage{cite}
\usepackage{hyperref}
\usepackage{enumerate}
\usepackage[left=0.7in,right=0.7in,top=0.7in,bottom=0.7in]{geometry}
\usepackage{titlesec}
\titleformat*{\section}{\large\bfseries}
\newtheorem{theorem}{Theorem}[section]
\newtheorem{lemma}[theorem]{Lemma}

\newtheorem{definition}[theorem]{Definition}
\newtheorem{example}[theorem]{Example}

\numberwithin{equation}{section}
\title{Inexact infinite products of weak quasi-contraction mappings in $b$-metric spaces}
\author{\large Anuradha Gupta and Manu Rohilla$^*$}
\date{}
\begin{document}
\maketitle
\begin{abstract} 
The influence of errors on the convergence of infinite products of weak quasi-contraction mappings in $b$-metric spaces is explored. An example demonstrating the necessity of convergence of the sequence of computational errors to zero is also provided. Moreover, we discuss weak ergodic theorems in the setting of $b$-metric spaces.


\end{abstract}    
\section{Introduction and Preliminaries}
The study of  convergence of iterations of nonexpansive mappings has always been a central topic in nonlinear analysis. Indeed, it is natural to ponder the  behavior of  the iterates of nonexpansive mappings in the presence of computational errors. Several authors (see \cite{3,4,9,10,11,12,13}) proved convergence results for iterates of nonexpansive mappings in the presence of computational errors in the context of Banach spaces and metric spaces. They have formulated necessary conditions on the sequence of computational errors for proving the convergence results.
  
Throughout this paper, we denote by $\mathbb{N}_0=\mathbb{N} \cup \{ 0\}$, where $\mathbb{N}$ denotes the set of natural numbers. The following definitions will be used in the sequel:
\begin{definition}
\emph{\cite{6} A $b$-metric on a nonempty set $X$ is a function $d:X \times X \rightarrow [0,\infty)$ such that for all $x,y,z \in X$ and a real number $s \geq 1$, the following conditions are satisfied:}

\emph{(i) $d(x,y)=0$ if and only if $x=y$,}

\emph{(ii) $d(x,y)=d(y,x)$,}

\emph{(iii) $d(x,y) \leq s[d(x,z)+d(z,y)]$. }\\
\emph{Then the pair $(X,d)$ is called a $b$-metric space. The number $s$ is called the coefficient of $(X,d)$.}
\end{definition}
\begin{definition}
\emph{\cite{2} Let $(X,d)$ be a $b$-metric space. Then}

\emph{(i) A sequence $\{x_n\} \subset X$ converges to $x \in X$ if and only if $\lim\limits_{n \rightarrow \infty}d(x_n,x)=0$.}

\emph{(ii) A sequence $\{x_n\} \subset X$ is said to be a Cauchy sequence if and only if $\lim\limits_{n,m \rightarrow \infty}d(x_n,x_m)=0$.}

\emph{(iii) A $b$-metric space $(X,d)$ is said to be complete if every Cauchy sequence $\{x_n\} \subset X$ converges to a point $x \in X$ such that $\lim\limits_{n \rightarrow \infty}d(x_n,x)=0=\lim\limits_{n,m \rightarrow \infty}d(x_n,x_m)$.}
\end{definition}
 Let $(X,d)$ be a $b$-metric space with coefficient $s \geq 1$. Let $T:X \rightarrow X$ be a mapping. Then the orbit and the double orbit induced by $T$ are defined by
\begin{align*}
\mathcal{O}_T(x):&= \{T^nx:n \in \mathbb{N}_0\},\\
\mathcal{O}_T(x,y):&= \mathcal{O}_T(x) \cup \mathcal{O}_T(y).
\end{align*}
By convention, $T^{n+1}=T \circ T^n$ and $T^0=I$, where $I:X \rightarrow X$ is the identity mapping. Recently, Bessenyei \cite{1} introduced the notion of weak quasi-contraction as follows:

 A mapping $T: X \rightarrow X$ is said to be a weak quasi-contraction if $T$ induces bounded orbits and
 $$d(Tx,Ty) \leq \psi (\mbox{diam } \mathcal{O}_T(x,y)) \thinspace \thinspace \mbox{for all } \thinspace x,y \in X,$$
 where $\psi:[0,\infty) \rightarrow [0,\infty)$ satisfies the following conditions:
 
 (i) $\psi$ is increasing and upper semi-continuous,
 
 (ii) $\psi(0)=0$,
 
 (iii) $\psi(t)<t$ for all $t>0$.
 
Mitrovi\'{c} and Hussain \cite{7} obtained fixed point results for weak quasi-contractions in the context of $b$-metric spaces. Let $C$ be a closed, bounded and convex subset of a Banach space endowed with a suitable complete metric.  Reich and Zaslavski \cite{12} obtained convergence results of infinite products of nonexpansive mappings on $C$. They \cite[Theorem 2.1]{12} proved that for a generic sequence $\{R_i\}_{i=1}^{\infty}$ in this space
$$\Vert R_nR_{n-1}\ldots R_2R_1x-R_nR_{n-1} \ldots R_2R_1y \Vert \rightarrow 0$$
an $n \rightarrow \infty$, uniformly for all $x,y \in X$. In the literature of population biology (see \cite{5,8}) such results are known as weak ergodic theorems.  They \cite[Theorem 2.2]{12}  proved the existence of a set $\mathfrak{R}$ which is a countable intersection of open and everywhere dense subsets of the space of sequences of nonexpansive mappings of $C$  such that for each mapping $g:\mathbb{N} \rightarrow \mathbb{N}$ and each  $\{S_i\}_{i=1}^{\infty} \in \mathfrak{R}$,
$$\Vert S_{g(n)}S_{g(n-1)}\ldots S_{g(2)}S_{g(1)}x-S_{g(n)}S_{g(n-1)}\ldots S_{g(2)}S_{g(1)}y \Vert \rightarrow 0$$
as $n \rightarrow \infty$, uniformly with respect to $g$ for all $x,y \in C$. Moreover, Butnariu et al. \cite{4} obtained convergence results for infinite products by assuming the convergence of exact infinite products and summability of errors. Pustylnik et al. \cite{11} proved that it is possible to establish the uniform  convergence of infinite products by only assuming that the computational errors to converge to $0$. Moreover, Reich and Zaslavski \cite{13} studied the convergence of infinite products of nonexpansive mappings in metric spaces by assuming the the convergence of uniform convergence of exact infinite orbits only on bounded subsets of metric space. Several authors (see \cite{3,9,10}) have studied the behavior of inexact orbits under the influence of computational errors and obtained convergence results.
 
 The main objective of the paper is to provide an  affirmative answer to the question of preservation of convergence of infinite products of weak quasi-contraction mappings in the setting of $b$-metric spaces. We obtain convergence results under the assumption that the exact infinite orbits converge and the sequence of computational errors converge to zero. Also, convergence results are established by assuming the convergence of exact infinite orbits on bounded subsets of the $b$-metric space. We provide an example to illustrate that convergence of the sequence of  computational errors to $0$ is necessary for establishing the convergence of inexact orbits. In the last section we formulate weak ergodic theorems in $b$-metric spaces.  
\section{Main Results}
Let $(X,d)$ be a $b$-metric space with coefficient $s \geq 1$. For each $x \in X$ and each nonempty subset $A \subset X$ define
$$d(x,A):=\inf \{d(x,y):y \in A\}.$$
Let $\mathfrak{F}$ be a nonempty set of mappings $f: \mathbb{N}_0 \rightarrow \mathbb{N}_0$ with the following property:

$(\mathcal{F})$ If $f \in \mathfrak{F}$ and $p \in \mathbb{N}$, then $f_p \in \mathfrak{F}$, where $f_p(i)=f(i+p)$ for all $i \in \mathbb{N}_0$.

\begin{theorem}\label{theorem1}
Let $(X,d)$ be a complete $b$-metric space with coefficient $s \geq 1$. Let $E$ be a nonempty and closed subset of $X$. 
For each $i \in \mathbb{N}_0$, let $T_i: X \rightarrow X$ be a  weak quasi-contraction mapping such that $T_i(E) \subset E$ for all $i \in \mathbb{N}_0$ and
\begin{align}
\mbox{diam }\mathcal{O}_{T_{f(j)}}(T_{f(i)}x,T_{f(i)}y) & \leq \mbox{diam }\mathcal{O}_{T_{f(j)}}(x,y) \thinspace \thinspace \mbox{for all } x,y \in X, \thinspace f \in \mathfrak{F} \thinspace \thinspace \mbox{and } i,j \in \mathbb{N}_0.\label{equation2}
\end{align}
Let $\mathfrak{F}$ be a nonempty set of mappings $f: \mathbb{N}_0 \rightarrow \mathbb{N}_0$  with property $(\mathcal{F})$. Suppose  that the following property holds:

$(\mathcal{P}_1)$ for each $\epsilon>0$, there exists a natural number $n_{\epsilon}$ such that for each $f \in \mathfrak{F}$ and each $x \in X$ we have
\begin{align*}
d(T_{f(n_{\epsilon})}T_{f(n_{\epsilon}-1)} \ldots T_{f(1)}T_{f(0)}x,E) < \epsilon.
\end{align*}
Then for each $\epsilon>0$, there exist $\delta>0$ and $n_0 \in \mathbb{N}$ such that for each $f \in \mathfrak{F}$ and each sequence $\{x_i\}_{i=0}^{\infty} \subset X$ satisfying 
\begin{align}\label{equation5}
\mbox{diam } \mathcal{O}_{T_{f(j)}}(x_{i+1},T_{f(i)}x_i) \leq \delta \thinspace \thinspace \mbox{for all } \thinspace i,j \in \mathbb{N}_0,
\end{align}
the following inequality is satisfied
$$d(x_i,E) < \epsilon \thinspace \thinspace \mbox{for all } \thinspace i \geq n_0.$$
\end{theorem}
\begin{proof}
Suppose that $\epsilon>0$ is given. Then by property $(\mathcal{P}_1)$, there exists $n_0 \in \mathbb{N}$ such that for each $f \in \mathfrak{F}$ and each $x \in X$ we have 
\begin{align}\label{equation6}
d(T_{f(n_0-2)}T_{f(n_0-3)} \ldots T_{f(1)}T_{f(0)}x,E) < \frac{\epsilon}{2s}.
\end{align}
Choose a real number $\delta$ such that
\begin{align}\label{equation7}
0<\delta < \frac{\epsilon}{2n_0s^{n_0}}.
\end{align}
Suppose that $f \in \mathfrak{F}$ and the sequence $\{x_i\}_{i=0}^{\infty} \subset X$ satisfies (\ref{equation5}). Let $n \geq n_0$ be an integer. It suffices to prove that $d(x_n,E) < \epsilon$. Set $z_i=x_{i+n-n_0+1}$ for all $i \in \mathbb{N}_0$. Then $z_{n_0-1}=x_n$. Let $\tilde{f}(i)=f(i+n-n_0+1)$ for all $i \in \mathbb{N}_0$. As $n\geq n_0$, by property $(\mathcal{F})$ we infer that $\tilde{f} \in \mathfrak{F}$.
For all $i, j \in \mathbb{N}_0$ using (\ref{equation5}) we have
\begin{align}
\mbox{diam }\mathcal{O}_{T_{\tilde{f}(j)}}(z_{i+1},T_{\tilde{f}(i)}z_i)& = \mbox{diam }\mathcal{O}_{T_{f(j+n-n_0+1)}}(x_{i+n-n_0+2},T_{f(i+n-n_0+1)}x_{i+n-n_0+1})\nonumber \\
&\leq \delta \label{equation8}.
\end{align}
Set $y_0=z_0$ and $y_{i+1}=T_{\tilde{f}(i)}y_i$ for all $i \in \mathbb{N}_0$. Then $y_{n_0-1}=T_{\tilde{f}(n_0-2)}T_{\tilde{f}(n_0-3)}\ldots T_{\tilde{f}(1)}T_{\tilde{f}(0)}z_0$. Since $\tilde{f} \in \mathfrak{F}$, using (\ref{equation6}) we deduce that
\begin{align}\label{equation9}
d(y_{n_0-1},E) < \frac{\epsilon}{2s}.
\end{align}
We claim that if $j \in \mathbb{N}_0$, then
\begin{align}\label{equation10}
\mbox{diam }\mathcal{O}_{T_{\tilde{f}(j)}}(z_i,y_i) \leq (i+1)s^i\delta  \thinspace \thinspace \thinspace \mbox{for all } \thinspace i \in \mathbb{N}_0.
\end{align}
We shall prove this by induction on $i$. If $i=0$, then using (\ref{equation5}) we have
\begin{align*}
\mbox{diam } \mathcal{O}_{T_{\tilde{f}(j)}}(z_0,y_0)&= \mbox{diam }\mathcal{O}_{T_{f(j+n-n_0+1)}}(x_{n-n_0+1})\\
& \leq \mbox{diam }\mathcal{O}_{T_{f(j+n-n_0+1)}}(x_{n-n_0+1},T_{f(n-n_0)}x_{n-n_0})\\
& \leq  \delta.
\end{align*}
Therefore, (\ref{equation10}) holds for $i=0$. Suppose that (\ref{equation10}) holds for some $i \in \mathbb{N}$. Now, we prove that it holds for $i+1$. Since the orbits are bounded,
$$\mbox{diam }\mathcal{O}_{T_{\tilde{f}(j)}}(z_{i+1},y_{i+1})= \sup_{k,l \in \mathbb{N}_0} \{ d(T_{\tilde{f}(j)}^k z_{i+1}, T_{\tilde{f}(j)}^l y_{i+1}),d(T_{\tilde{f}(j)}^k z_{i+1},T_{\tilde{f}(j)}^l z_{i+1}),d(T_{\tilde{f}(j)}^k y_{i+1},T_{\tilde{f}(j)}^l y_{i+1}) \}.$$
Consider  
\begin{align*}
d(T_{\tilde{f}(j)}^k z_{i+1},T_{\tilde{f}(j)}^l y_{i+1})& \leq s d(T_{\tilde{f}(j)}^k z_{i+1}, T_{\tilde{f}(j)}^k T_{\tilde{f}(i)}z_i)+sd(T_{\tilde{f}(j)}^k T_{\tilde{f}(i)}z_i,T_{\tilde{f}(j)}^l y_{i+1})\\
& \leq s \psi^k(\mbox{diam } \mathcal{O}_{T_{\tilde{f}(j)}}(z_{i+1},T_{\tilde{f}(i)}z_i))+s \thinspace \mbox{diam } \mathcal{O}_{T_{\tilde{f}(j)}}(T_{\tilde{f}(i)}z_i,y_{i+1})\\
&< s \thinspace \mbox{diam } \mathcal{O}_{T_{\tilde{f}(j)}}(z_{i+1},T_{\tilde{f}(i)} z_i)+s \thinspace \mbox{diam } \mathcal{O}_{T_{\tilde{f}(j)}}(T_{\tilde{f}(i)}z_i,T_{\tilde{f}(i)}y_i).
\end{align*}
Using (\ref{equation2}) we have
\begin{align*}
d(T_{\tilde{f}(j)}^k z_{i+1},T_{\tilde{f}(j)}^l y_{i+1})& \leq s \thinspace \mbox{diam } \mathcal{O}_{T_{\tilde{f}(j)}}(z_{i+1},T_{\tilde{f}(i)} z_i)+s \thinspace \mbox{diam } \mathcal{O}_{T_{\tilde{f}(j)}}(z_i,y_i).  
\end{align*}
Using (\ref{equation8}) and the induction hypothesis we get,
$$d(T_{\tilde{f}(j)}^k z_{i+1},T_{\tilde{f}(j)}^l y_{i+1}) < s \delta+(i+1)s^{i+1}\delta \leq (i+2)s^{i+1}\delta.$$
Now using (\ref{equation8}) we have
\begin{align*}
d(T_{\tilde{f}(j)}^kz_{i+1},T_{\tilde{f}(j)}^l z_{i+1}) & \leq \mbox{diam } \mathcal{O}_{T_{\tilde{f}(j)}}(z_{i+1}) \leq \mbox{diam } \mathcal{O}_{T_{\tilde{f}(j)}}(z_{i+1},T_{\tilde{f}(i)}z_i) \leq \delta.
\end{align*}
Also, using (\ref{equation2}) and the induction hypothesis we have
\begin{align*}
d(T_{\tilde{f}(j)}^ky_{i+1},T_{\tilde{f}(j)}^l y_{i+1})& \leq \mbox{diam } \mathcal{O}_{T_{\tilde{f}(j)}} (y_{i+1})=\mbox{diam } \mathcal{O}_{T_{\tilde{f}(j)}}(T_{\tilde{f}(i)}y_i) \leq \mbox{diam } \mathcal{O}_{T_{\tilde{f}(j)}}(z_i,y_i)\leq (i+1)s^i \delta.
\end{align*}
This gives $\mbox{diam } \mathcal{O}_{T_{\tilde{f}(j)}}(z_{i+1},y_{i+1}) \leq (i+2)s^{i+1} \delta$. Therefore, the claim follows. This implies that 
\begin{align}\label{equation11}
\mbox{diam } \mathcal{O}_{T_{\tilde{f}(j)}}(z_{n_0-1},y_{n_0-1}) \leq n_0 s^{n_0-1} \delta.
\end{align}
By (\ref{equation9}) and (\ref{equation11}) we deduce that
\begin{align*}
d(x_n,E)=d(z_{n_0-1},E)& \leq sd(z_{n_0-1},y_{n_0-1})+s d(y_{n_0-1},E)\\
& \leq s \thinspace \mbox{diam } \mathcal{O}_{T_{\tilde{f}(j)}}(z_{n_0-1},y_{n_0-1})+s d(y_{n_0-1},E)\\
&\leq n_0s^{n_0}\delta+ \frac{\epsilon}{2}.
\end{align*}
By (\ref{equation7}) it follows that $d(x_n,E) < \epsilon$. This completes the proof.
\end{proof}
\begin{theorem}\label{theorem2}
Let $(X,d)$ be a complete $b$-metric space with coefficient $s \geq 1$. Let $E$ be a nonempty and closed subset of $X$. For each $i \in \mathbb{N}_0$, let $T_i: X \rightarrow X$ be a weak quasi-contraction mapping such that $T_i(E) \subset E$ for all $i \in \mathbb{N}_0$ and $(\ref{equation2})$ holds.  
Let $\mathfrak{F}$ be a nonempty set of mappings $f: \mathbb{N}_0 \rightarrow \mathbb{N}_0$ with  property $(\mathcal{F})$. Suppose  that property $(\mathcal{P}_1)$ holds. Let $\{\delta_i \}_{i=0}^{\infty}$ be a sequence of positive numbers such that  $\lim\limits_{i \rightarrow \infty} \delta_i=0$.
Then for each $\epsilon>0$, there exists $\overline{n} \in \mathbb{N}$ such that for each $f \in \mathfrak{F}$ and each sequence $\{x_i\}_{i=0}^{\infty} \subset X$ satisfying 
\begin{align}\label{equation16}
\mbox{diam } \mathcal{O}_{T_{f(j)}}(x_{i+1},T_{f(i)}x_i) \leq \delta_i \thinspace \thinspace \mbox{for all } \thinspace i,j \in \mathbb{N}_0,
\end{align}
the following inequality is satisfied
$$d(x_i,E) < \epsilon \thinspace \thinspace \mbox{for all } \thinspace i \geq \overline{n}.$$
\end{theorem}
\begin{proof}
By Theorem \ref{theorem1} there exist $\delta>0$ and $n_0 \in \mathbb{N}$ such that the following property is satisfied:

$(\mathcal{P}_2)$  for each $f \in \mathfrak{F}$ and each sequence $\{x_i\}_{i=0}^{\infty} \subset X$ such that (\ref{equation16}) holds, the following inequality is satisfied
$$d(x_i,E) < \epsilon \thinspace \thinspace \mbox{for all } \thinspace i \geq n_0.$$
Since $\lim\limits_{i \rightarrow \infty} \delta_i=0$, there exists $n_1 \in \mathbb{N}$ such that 
\begin{align}\label{equation17}
\delta_i < \delta \thinspace \thinspace \mbox{for all } \thinspace i \geq n_1.
\end{align}
Let $n'=\max \{n_0,n_1\}$ and $\overline{n} \geq 2n_0+n_1$. Then $\overline{n} \geq n'$. Suppose that $f \in \mathfrak{F}$ and the sequence $\{x_i \}_{i=0}^{\infty} \subset X$ satisfies (\ref{equation16}). Let $n \geq \overline{n}$ be an integer. It suffices to prove that $d(x_n,E)< \epsilon$. Set $z_i=x_{i+n'}$ and $\overline{f}(i)=f(i+n')$ for all $i \in \mathbb{N}_0$. Since $n' \in \mathbb{N}$, by property $(\mathcal{F})$ we get $\overline{f} \in \mathfrak{F}$. In view of (\ref{equation16}) and (\ref{equation17}),
\begin{align*}
\mbox{diam } \mathcal{O}_{T_{\overline{f}(j)}}(z_{i+1},T_{\overline{f}(i)} z_i)&= \mbox{diam } \mathcal{O}_{T_{f(j+n')}}(x_{i+n'+1},T_{f(i+n')}x_{i+n'})\\
& \leq \delta_{i+n'}\\
& < \delta.
\end{align*}
By using property $(\mathcal{P}_2)$ we infer that
\begin{align}\label{equation18}
d(z_i,E)< \epsilon \thinspace \thinspace \mbox{for all } \thinspace i \geq n_0.
\end{align}
Since $n \geq \overline{n}$ and $n'=\max \{n_0,n_1\}$, $n-n' \geq n_0$. Therefore, from (\ref{equation18}) it follows that
$$d(x_n,E)=d(z_{n-n'},E) < \epsilon.$$
\end{proof}
If we consider the set $E$ to be a singleton, then we have the following result:
\begin{theorem}
Let $E=\{x\}$ and all the assumptions of Theorem $\ref{theorem2}$ are satisfied. Then for each $\epsilon>0$, there exists $\breve{n} \in \mathbb{N}$ such that for each $f \in \mathfrak{F}$ and each sequence $\{x_i\}_{i=0}^{\infty} \subset X$ satisfying $(\ref{equation16})$, the following inequality is satisfied
$$d(x_i,x) < \epsilon \thinspace \thinspace \mbox{for all } \thinspace i \geq \breve{n}.$$ 
\end{theorem}
Now we construct an example which shows that if the sequence of errors does not convergence to $0$, then there exist  inexact orbits which do not converge.
 
Let $X$ be the set of all sequences $x=\{x_i\}_{i=1}^{\infty}$ of nonnegative numbers such that $\sum\limits_{i=1}^{\infty} x_i \leq 1$. For $x=\{x_i\}_{i=1}^{\infty}$, $y=\{y_i\}_{i=1}^{\infty} \in X$, set 
\begin{align}\label{eq}
d(x,y)= \sum\limits_{i=1}^{\infty} \vert x_i-y_i \vert^2.
\end{align}
Then $(X,d)$ is a complete $b$-metric space with coefficient $s=2$. Define $T:X \rightarrow X$ by
$$Tx=T(\{x_i\}_{i=1}^{\infty})=\Big(\frac{x_2}{2},\frac{x_3}{2},\frac{x_4}{2}, \ldots \Big).$$
Set $T^{0}x=x$ for all $x \in X$. Let $\mathfrak{F}$ be the set of all mappings $f: \mathbb{N}_0 \rightarrow \mathbb{N}_0$ and $T_i=T$ for all $i \in \mathbb{N}_0$. Define $\psi:[0,\infty) \rightarrow [0,\infty)$ by $\psi(t)=\frac{t}{3}$. Then $d(Tx,Ty) \leq \psi(\mbox{diam } \mathcal{O}_T(x,y))$ and $\mbox{diam } \mathcal{O}_T(Tx,Ty) \leq \mbox{diam } \mathcal{O}_T(x,y)$ for all $x,y \in X$. The following lemma is instrumental in constructing the example:
\begin{lemma}\label{lemma}
Let $w^{(0)}=\{w_i^{(0)}\}_{i=1}^{\infty} \in X$ be such that $\mbox{diam } \mathcal{O}_T(w^{(0)}) \leq \frac{1}{3}$ and $q \in \mathbb{N}_0$. Let $\{\gamma_i\}_{i=0}^{\infty}$ be a sequence of positive numbers  such that $\lim\limits_{i \rightarrow \infty}\gamma_i \neq 0$.  Then there exist a natural number $n \geq 4$ and a sequence $\{w^{(i)}\}_{i=0}^n \subset X$ such that
\begin{align*}
\mbox{diam } \mathcal{O}_T(w^{(i)}) &\leq \frac{1}{3} \thinspace \thinspace \mbox{for all } \thinspace i \in \{0,1,\ldots,n\},\\
\mbox{diam } \mathcal{O}_T(w^{(i+1)},Tw^{(i)})& \leq \gamma_{q+i}+\frac{1}{3} \thinspace \thinspace \mbox{for all } \thinspace i \in \{0,1,\ldots,n-1\},\\
d(w^{(n)},(0,0,\ldots))& \geq 5 \times 10^{-5}. 
\end{align*}
\end{lemma}
\begin{proof}
Without loss of generality we may assume that
\begin{align}\label{eq1}
\gamma_i \leq \frac{1}{64} \thinspace \thinspace \mbox{for all } \thinspace i \in \mathbb{N}_0.
\end{align}
There is a natural number $m>4$ such that
\begin{align}\label{eq2}
\sum\limits_{i=m}^{\infty} w_i^{(0)} \leq \frac{1}{32}.
\end{align}
Set $w^{(i+1)}=Tw^{(i)}$ for all $i \in \{0,1,\ldots, m-1\}$. Then 
$$w^{(m)}= \Big( \frac{w_{m+1}^{(0)}}{2^m},\frac{w_{m+2}^{(0)}}{2^m}, \frac{w_{m+3}^{(0)}}{2^m}, \ldots \Big).$$
Since $\lim\limits_{j \rightarrow \infty} \gamma_j\neq 0$, $\sum\limits_{j=0}^{\infty} \gamma_j = \infty$. Then there exists a natural number $n>m$ such that
\begin{align}\label{eq3}
\sum\limits_{j=q+m}^{q+n} \gamma_j \geq \frac{1}{32}.
\end{align}
By (\ref{eq1}) and (\ref{eq3}), $n \geq m+1$ and  without loss of generality we may assume that 
\begin{align}\label{eq4}
\sum_{j=q+m}^{q+n-1} \gamma_j < \frac{1}{32}.
\end{align}
Using (\ref{eq1}) and (\ref{eq3}) we have 
\begin{align}\label{eq5}
\sum\limits_{j=q+m}^{q+n-1} \gamma_j = \sum\limits_{j=q+m}^{q+n} \gamma_j -\gamma_{q+n} \geq \frac{1}{32}-\frac{1}{64}=\frac{1}{64}.
\end{align}
For $i \in \{m+1,m+2,\ldots,n\}$, define $w^{(i)}=\{w_j^{(i)} \}_{j=1}^{\infty}$ as follows:
\begin{align}
w_j^{(i)}&=\frac{w_{j+i}^{(0)}}{2^i} \thinspace \thinspace \mbox{for all } \thinspace j \in \mathbb{N} \setminus \{n+1-i\}, \nonumber\\
w_{n+1-i}^{(i)}&= \frac{w_{n+1}^{(0)}}{2^i}+\frac{1}{2} \sum\limits_{j=q+m}^{q+i-1} \gamma_j. \label{neweq2}
\end{align}
Evidently, $w^{(i)}$ is well defined for all $i \in \{m+1,m+2,\ldots,n\}$. By (\ref{eq2}) and (\ref{eq4}),
\begin{align*}
\sum_{j=1}^{\infty} w_j^{(i)}&= \frac{1}{2^i} \sum\limits_{j=i+1}^{\infty} w_j^{(0)}+ \frac{1}{2} \sum\limits_{j=q+m}^{q+i-1} \gamma_j\\
& \leq \sum\limits_{j=m}^{\infty} w_j^{(0)}+ \frac{1}{2} \sum\limits_{j=q+m}^{q+n-1} \gamma_j\\
& < \frac{3}{64} \leq 1
\end{align*}
which implies that $w^{(i)} \in X$ for all $i \in \{m+1,m+2,\ldots,n\}$. If $0 \leq i \leq m$, then
\begin{align}
\mbox{diam } \mathcal{O}_T(w^{(i)})&= \mbox{diam } \mathcal{O}_T(Tw^{(i-1)}) \nonumber\\
& \leq \mbox{diam } \mathcal{O}_T(w^{(i-1)}) \nonumber\\
& \quad \vdots \nonumber\\
& \leq \mbox{diam } \mathcal{O}_T(w^{(0)}) \nonumber\\
& \leq \frac{1}{3}. \label{neweq1}
\end{align}
Also, if $i \in \{m+1,m+2,\ldots,n\}$, then $\mbox{diam } \mathcal{O}_T(w^{(i)}) \leq \frac{1}{3}$. Now, we show that $\mbox{diam } \mathcal{O}_T(w^{(i+1)},Tw^{(i)}) \leq \gamma_{q+i}+\frac{1}{3}$ for all $i \in \{0,1,\ldots,n-1\}$. If $i \in \{0,1,\ldots,m-1\}$, then using (\ref{neweq1}) 
$$\mbox{diam } \mathcal{O}_T(w^{(i+1)},Tw^{(i)})= \mbox{diam } \mathcal{O}_T(w^{(i+1)}) \leq \frac{1}{3}\leq \gamma_{q+i}+\frac{1}{3}.$$
If $i \in \{m,m+1,\ldots,n-1\}$, then
$$\mbox{diam } \mathcal{O}_T(w^{(i+1)},Tw^{(i)}) = \sup\limits_{k, l \in \mathbb{N}_0} \{d(T^k w^{(i+1)},T^l(Tw^{(i)})),d(T^k w^{(i+1)},T^l w^{(i+1)}),d(T^k(Tw^{(i)}),T^l(Tw^{(i)}))\}.$$
We have
\[
w^{(i+1)}= \Big( \frac{w_{2+i}^{(0)}}{2^{i+1}},\frac{w_{3+i}^{(0)}}{2^{i+1}},\ldots, \frac{w_n^{(0)}}{2^{i+1}}, \underbrace{\frac{w_{n+1}^{(0)}}{2^{i+1}}+\frac{1}{2} \sum\limits_{j=q+m}^{q+i} \gamma_j}_\text{$(n-i)^{\mbox{th}}$ position}, \frac{w_{n+2}^{(0)}}{2^{i+1}}, \ldots \Big)
\]
and 
\[
Tw^{(i)}= \Big( \frac{w_{2+i}^{(0)}}{2^{i+1}},\frac{w_{3+i}^{(0)}}{2^{i+1}},\ldots, \frac{w_n^{(0)}}{2^{i+1}}, \underbrace{\frac{w_{n+1}^{(0)}}{2^{i+1}}+\frac{1}{4} \sum\limits_{j=q+m}^{q+i-1} \gamma_j}_\text{$(n-i)^{\mbox{th}}$ position}, \frac{w_{n+2}^{(0)}}{2^{i+1}}, \ldots \Big).
\]
Therefore, using (\ref{eq}) and (\ref{eq4}) we have 
\begin{align*}
d(w^{(i+1)},Tw^{(i)})&= \Big\vert \frac{1}{2} \gamma_{q+i}+\frac{1}{4} \sum\limits_{j=q+m}^{q+i-1} \gamma_j \Big\vert^2\\
& \leq \gamma_{q+i}^2+\frac{1}{4} \Big( \sum\limits_{j=q+m}^{q+n-1} \gamma_j \Big)^2\\
& < \gamma_{q+i}+\frac{1}{4} \Big( \frac{1}{32} \Big)^2\\
& \leq \gamma_{q+i}+\frac{1}{3}.
\end{align*}
Similarly, $d(T^k w^{(i+1)},T^l(Tw^{(i)})) \leq \gamma_{q+i}+\frac{1}{3}$, for all $k,l \in \mathbb{N}_0$. Also, for all $k,l \in \mathbb{N}_0$ 
\begin{align*}
d(T^kw^{(i+1)},T^lw^{(i+1)})& \leq \mbox{diam } \mathcal{O}_T(w^{(i+1)}) \leq \frac{1}{3},\\
d(T^k(Tw^{(i)}),T^l(Tw^{(i)}))& \leq \mbox{diam } \mathcal{O}_T(w^{(i)}) \leq \frac{1}{3}.
\end{align*}
Therefore, $\mbox{diam } \mathcal{O}_T(w^{(i+1)},Tw^{(i)}) \leq \gamma_{q+i}+\frac{1}{3}$ for all $i \in \{0,1,2,\ldots,n-1\}$. By (\ref{eq}) and (\ref{eq5}), (\ref{neweq2})
\begin{align*}
d(w^{(n)},(0,0,\ldots))&= \sum\limits_{j=1}^{\infty} (w_j^{(n)})^2\\
& \geq (w_1^{(n)})^2\\
&=(w_{n+1-n}^{(n)})^2\\
&= \Big( \frac{w_{n+1}^{(0)}}{2^n}+\frac{1}{2} \sum\limits_{j=q+m}^{q+n-1} \gamma_j \Big)^2\\
& \geq \frac{1}{4} \Big( \sum\limits_{j=q+m}^{q+n-1} \gamma_j \Big)^2\\
& \geq \frac{1}{4} \Big( \frac{1}{64} \Big)^2\\
& \geq 5 \times 10^{-5}.
\end{align*}  
\end{proof}
Let $\delta_i=\gamma_i+\frac{1}{3}$ for all $i \in \mathbb{N}_0$. Since $\lim\limits_{i \rightarrow \infty}\gamma_i \neq 0$, $\lim\limits_{i \rightarrow \infty} \delta_i \neq 0$.
\begin{example}
Let $\{\delta_i\}_{i=0}^{\infty}$ be a sequence of positive numbers such that $\lim\limits_{i \rightarrow \infty} \delta_i \neq 0$. Let $x=\{x_i\}_{i=1}^{\infty} \in X$ such that $\mbox{diam } \mathcal{O}_T(x) \leq \frac{1}{3}$. Then there exist a  sequence $\{z^{(i)} \}_{i=0}^{\infty} \subset X$ such that $z^{(0)}=x$, 
\begin{align*}
\mbox{diam } \mathcal{O}_T(z^{(i)}) & \leq \frac{1}{3} \thinspace \thinspace \mbox{for all } \thinspace i \in \mathbb{N}_0,\\
\mbox{diam } \mathcal{O}_T(z^{(i+1)},Tz^{(i)})& \leq \delta_i \thinspace \thinspace \mbox{for all } \thinspace i \in \mathbb{N}_0 
\end{align*}
and a sequence of nonnegative integers $\{t_k\}_{k=0}^{\infty}$ such that
$$d(z^{(t_k)},(0,0,\ldots)) \geq 5 \times 10^{-5}\thinspace \thinspace \mbox{for all } \thinspace k \in \mathbb{N}.$$
\end{example}
\begin{proof}
Using induction we construct a sequence $\{z^{(i)}\}_{i=0}^{\infty} \subset X$ and a sequence of nonnegative integers $\{t_k\}_{k=0}^{\infty}$ such that
\begin{equation}\label{neweq3}
\left\{
\begin{aligned}
z^{(0)}&=x,\\
\mbox{diam } \mathcal{O}_T(z^{(i)}) & \leq \frac{1}{3} \thinspace \thinspace \mbox{for all } \thinspace i \in \mathbb{N}_0,\\
\mbox{diam } \mathcal{O}_T(z^{(i+1)},Tz^{(i)}) & \leq \delta_i \thinspace \thinspace \mbox{for all } \thinspace i \in \mathbb{N}_0,\\
t_0=0, \thinspace t_k&<t_{k+1} \thinspace \thinspace \mbox{for all } \thinspace k \in \mathbb{N}_0,\\
d(z^{(t_k)},(0,0,\ldots))& \geq 5 \times 10^{-5}. 
\end{aligned} \right.
\end{equation}
Set $z^{(0)}=x$ and $t_0=0$. Suppose that the result holds for some $p \in \mathbb{N}$. Then we have already defined a sequence $\{z^{(i)}\}_{i=0}^{t_p} \subset X$ and a sequence of nonnegative integers $\{t_k\}_{k=0}^p$ such that
\begin{align*}
z^{(0)}&=x\\
\mbox{diam } \mathcal{O}_T(z^{(i)}) & \leq \frac{1}{3} \thinspace \thinspace \mbox{for all } \thinspace i \in \{0,1,\ldots,t_p\},\\
\mbox{diam } \mathcal{O}_T(z^{(i+1)},Tz^{(i)}) & \leq \delta_i \thinspace \thinspace \mbox{for all } \thinspace i \in \{0,1,\ldots,t_p-1\},\\
t_0=0, \thinspace t_k &< t_{k+1} \thinspace \thinspace \mbox{for all } \thinspace k \in \{0,1,\ldots,p-1\},\\
d(z^{(t_k)},(0,0,\ldots)) & \geq 5 \times 10^{-5} \thinspace \thinspace \mbox{for all } \thinspace k \in \{1,2,\ldots,p\}. 
\end{align*}
Now, we show that this assumption holds for $p+1$. Applying Lemma \ref{lemma} with $w^{(0)}=z^{(t_p)}$ and $q=t_p$, there exist a natural number $n \geq 4$ and a sequence $\{z^{(i)}\}_{i=t_p}^{t_p+n} \subset X$ such that
\begin{align*}
\mbox{diam } \mathcal{O}_T(z^{(i)}) & \leq \frac{1}{3} \thinspace \thinspace \mbox{for all } \thinspace i \in \{t_p,t_{p}+1,\ldots,t_p+n\},\\
\mbox{diam } \mathcal{O}_T(z^{(i+1)},Tz^{(i)})& \leq \delta_i \thinspace \thinspace \mbox{for all } \thinspace i \in \{t_p,t_{p}+1,\ldots,t_p+n-1\},\\
d(z^{(t_p+n)},(0,0,\ldots))& \geq 5 \times 10^{-5}.
\end{align*}
Put $t_{p+1}=t_p+n$. Then $t_p<t_{p+1}$. In this way, we have constructed a sequence $\{z^{(i)}\}_{i=0}^{t_{p+1}} \subset X$ and a sequence of nonnegative integers $\{t_k\}_{k=0}^{p+1}$ such that
\begin{align*}
z^{(0)}&=x,\\
\mbox{diam } \mathcal{O}_T(z^{(i)}) & \leq \frac{1}{3} \thinspace \thinspace \mbox{for all } \thinspace i \in \{0,1,\ldots,t_{p+1}\},\\
\mbox{diam } \mathcal{O}_T(z^{(i+1)},Tz^{(i)}) & \leq \delta_i \thinspace \thinspace \mbox{for all } \thinspace i \in \{0,1,\ldots, t_{p+1}-1\},\\
t_0=0, \thinspace t_k &< t_{k+1} \thinspace \thinspace \mbox{for all } \thinspace k \in \{0,1,\ldots,p\},\\
d(z^{(t_k)},(0,0,\ldots))& \geq 5 \times 10^{-5} \thinspace \thinspace \mbox{for all } \thinspace k \in \{1,2,\ldots,p+1\}.
\end{align*}
Therefore, the assumption holds for $p+1$. This implies that we have constructed a sequence $\{z^{(i)}\}_{i=0}^{\infty} \subset X$ and a sequence of nonnegative integers $\{t_k\}_{k=0}^{\infty}$ satisfying (\ref{neweq3}).
\end{proof} 
Now we establish the convergence results of inexact orbits by assuming the convergence of exact orbits only on bounded subsets of a $b$-metric space $(X,d)$.
\begin{definition}
\emph{Let $(X,d)$ be a $b$-metric space. A mapping $T: X \rightarrow X$ is said to be $\mathcal{O}$-$b$ continuous if for a sequence $\{x_n\} \subset X$, $x_n \rightarrow x$ in $(X,d)$ implies that $\mbox{diam } \mathcal{O}_T(x_n,x) \rightarrow 0$.}
\end{definition}
\begin{theorem}\label{theorem3}
Let $(X,d)$ be a complete $b$-metric space with coefficient $s \geq 1$. Let $E$ be a nonempty, bounded and closed subset of $X$. For each $i \in \mathbb{N}_0$, let $T_i:X \rightarrow X$ be a weak quasi-contraction mapping, $\mathcal{O}$-$b$ continuous and 
\begin{align}
\mbox{diam } \mathcal{O}_{T_j}(T_ix,T_iy) & \leq \mbox{diam } \mathcal{O}_{T_j}(x,y) \thinspace \thinspace \mbox{for all } \thinspace x,y \in X, \thinspace f \in \mathfrak{F} \thinspace \thinspace \mbox{and } \thinspace i,j \in \mathbb{N}_0. \label{equation20}
\end{align}
Let $\mathfrak{F}$ be a nonempty set of mappings $f: \mathbb{N}_0 \rightarrow \mathbb{N}_0$ with property $(\mathcal{F})$. Suppose  that the following properties holds:

$(\mathcal{P}_3)$ for each $e \in E$ and each $i \in \mathbb{N}_0$, there exists $e' \in E$ such that $T_i(e')=e$.

$(\mathcal{P}_4)$ for each $\epsilon,K>0$, there exists a natural number $n_{\epsilon,K}$ such that for each $f \in \mathfrak{F}$ and each $x \in X$ we have			
\begin{align*}
\mbox{diam } \mathcal{O}_{T_i}(x, \theta) \leq \frac{K}{s} \thinspace \thinspace \mbox{for all } \thinspace i \in \mathbb{N}_0,\\
d(T_{f(n_{\epsilon,K})}T_{f(n_{\epsilon, K}-1)} \ldots T_{f(1)}T_{f(0)}x,E) < \epsilon. 
\end{align*}
Then for each $\epsilon,K>0$, there exist $\delta>0$ and $\hat{n} \in \mathbb{N}$ such that for each $f \in \mathfrak{F}$ and each sequence $\{x_i \}_{i=0}^{\infty} \subset X$ satisfying
\begin{align}\label{equation23}
\mbox{diam } \mathcal{O}_{T_i}(x_0,\theta) \leq \frac{K}{s} \thinspace \thinspace \mbox{for all } \thinspace i \in \mathbb{N}_0
\end{align}
and
\begin{align}\label{equation24}
\mbox{diam } \mathcal{O}_{T_{f(j)}}(x_{i+1},T_{f(i)}x_i) \leq \delta \thinspace \thinspace \mbox{for all } \thinspace i,j \in \mathbb{N}_0,
\end{align}
the following inequality is satisfied
$$d(x_i,E)< \epsilon \thinspace \thinspace \mbox{for all } \thinspace i \geq \hat{n}.$$
\end{theorem}
\begin{proof}
Suppose that $\epsilon,K>0$ are given. Without loss of generality we may assume that
\begin{align}\label{equation25}
K>2 \thinspace \thinspace \mbox{and } \thinspace \mbox{diam } \mathcal{O}_{T_i}(\theta, e) \leq \frac{K-2}{s} \thinspace \thinspace \mbox{for all } \thinspace i \in \mathbb{N}_0 \thinspace \thinspace \mbox{and } \thinspace e \in E.
\end{align}
Let 
$$Z= \{ x \in X: \mbox{diam } \mathcal{O}_{T_{f(i)}}(x,e) \leq s^k (2K-1) \thinspace \thinspace \mbox{for some } \thinspace k \in \mathbb{N} \thinspace \thinspace \mbox{and for all } \thinspace f \in \mathfrak{F}, i \in \mathbb{N}_0 \thinspace \mbox{ and } \thinspace e \in E \}.$$
We shall show that $(Z,d)$ is a complete $b$-metric space and $T_j(Z) \subset Z$ for all $j \in \mathbb{N}_0$. Let $\{x_n\}$ be  a Cauchy sequence in $(Z,d)$. As $(X,d)$ is a complete $b$-metric space, there exists $x \in X$ such that $\lim\limits_{n \rightarrow \infty}d(x_n,x)=0$. Since the orbits are bounded,
$$\mbox{diam } \mathcal{O}_{T_{f(i)}}(x,e)= \sup\limits_{l,m \in \mathbb{N}_0} \{ d(T_{f(i)}^lx, T_{f(i)}^me),d(T_{f(i)}^lx,T_{f(i)}^mx),d(T_{f(i)}^le,T_{f(i)}^me)\}.$$
Consider
\begin{align*}
d(T_{f(i)}^lx,T_{f(i)}^me)& \leq sd(T_{f(i)}^lx, T_{f(i)}^lx_n)+sd(T_{f(i)}^lx_n,T_{f(i)}^me)\\
&\leq s \thinspace \mbox{diam } \mathcal{O}_{T_{f(i)}}(x,x_n)+s \thinspace \mbox{diam } \mathcal{O}_{T_{f(i)}}(x_n,e).
\end{align*}
Since $T_i$ is $\mathcal{O}$-$b$ continuous and $x_n \rightarrow x$ in $(X,d)$, $\mbox{diam } \mathcal{O}_{T_i}(x_n,x) \rightarrow 0$ in $(X,d)$. Also, $x_n \in Z$ implies that $\mbox{diam } \mathcal{O}_{T_{f(i)}}(x_n,e) \leq s^k(2K-1)$ for some $k \in \mathbb{N}$. This implies that
$$d(T_{f(i)}^lx,T_{f(i)}^me) \leq s^{k+1}(2K-1).$$
Now consider 
\begin{align*}
d(T_{f(i)}^lx, T_{f(i)}^mx)& \leq \mbox{diam } \mathcal{O}_{T_{f(i)}}(x)\\
& \leq \mbox{diam } \mathcal{O}_{T_{f(i)}}(x,x_n) \rightarrow 0 \thinspace \thinspace \mbox{as } \thinspace n \rightarrow \infty.
\end{align*} 
Also, consider
\begin{align*}
d(T_{f(i)}^le,T_{f(i)}^me) & \leq \mbox{diam } \mathcal{O}_{T_{f(i)}}(e) \leq \mbox{diam } \mathcal{O}_{T_{f(i)}}(x_n,e) \leq s^k(2K-1).
\end{align*}
Therefore, for all $f \in \mathfrak{F}$, $i \in \mathbb{N}_0$ and $e \in E$ we have
$$\mbox{diam } \mathcal{O}_{T_{f(i)}}(x,e) \leq s^{k+1}(2K-1)$$
which implies that $x \in Z$. Now we prove that $T_j(Z) \subset Z$. Let $x \in Z$, then $\mbox{diam } \mathcal{O}_{T_{f(i)}}(x,e) \leq s^k(2K-1)$ for some $k \in \mathbb{N}$ and for all $f \in \mathfrak{F}$, $i \in \mathbb{N}_0$ and $e \in E$. Since $e \in E$, by property $(\mathcal{P}_3)$ for each $j \in \mathbb{N}_0$ there exists $e' \in E$ such that $T_je'=e$. Therefore, by (\ref{equation20})
\begin{align*}
\mbox{diam } \mathcal{O}_{T_{f(i)}}(T_jx,e)&= \mbox{diam } \mathcal{O}_{T_{f(i)}}(T_jx,T_je')\\
& \leq \mbox{diam } \mathcal{O}_{T_{f(i)}}(x,e')\\
& \leq s^{k'}(2K-1),
\end{align*}
for some $k' \in \mathbb{N}$ which gives $T_jx \in Z$. Therefore, $T_j(Z) \subset Z$ for all $j \in \mathbb{N}_0$. Thus, all the assumptions of Theorem \ref{theorem1} are satisfied for the complete $b$-metric space $(Z,d)$ and the restrictions of $T_i$ to $Z$, $i \in \mathbb{N}_0$. Therefore, by Theorem \ref{theorem1} there exist $\delta>0$ and $\hat{n} \in \mathbb{N}$ such that the following property is satisfied:

$(\mathcal{P}_5)$ for each $f \in \mathfrak{F}$ and each sequence $\{x_i\}_{i=0}^{\infty} \subset Z$ such that
$$\mbox{diam } \mathcal{O}_{T_{f(j)}}(x_{i+1},T_{f(i)}x_i) \leq \delta \thinspace \thinspace \mbox{for all } \thinspace i,j \in \mathbb{N}_0,$$
the following inequality is satisfied
$$d(x_i,E) < \epsilon \thinspace \thinspace \mbox{for all } \thinspace i \geq \hat{n}.$$
Without loss of generality we may assume that
\begin{align}\label{equation26}
\hat{n} \delta<1.
\end{align}
We claim that the following property is satisfied:

$(\mathcal{P}_6)$ if $f \in \mathfrak{F}$ and a sequence $\{x_i\}_{i=0}^{\hat{n}} \subset X$ satisfies
\begin{align}\label{equation27}
\mbox{diam } \mathcal{O}_{T_i}(x_0,e) \leq 2K-2 \thinspace \thinspace \mbox{for all } \thinspace i \in \mathbb{N}_0, e \in E
\end{align}
and for all $i\in \{0,1,\ldots,\hat{n}-1\}$
\begin{align}\label{equation28}
\mbox{diam } \mathcal{O}_{T_{f(j)}}(x_{i+1},T_{f(i)}x_i) \leq \delta \thinspace \thinspace \mbox{for all } \thinspace j \in \mathbb{N}_0,
\end{align}
then $\{x_i\}_{i=0}^{\hat{n}} \subset Z$.  

Suppose that $f \in \mathfrak{F}$ and the sequence $\{x_i\}_{i=0}^{\hat{n}} \subset X$ satisfies (\ref{equation27}) and (\ref{equation28}). Then for each $i\in \{0,1, \ldots,\hat{n}-1\}$ 
$$\mbox{diam } \mathcal{O}_{T_{f(j)}}(x_{i+1},e)= \sup\limits_{l,m \in \mathbb{N}_0} \{d(T_{f(j)}^lx_{i+1},T_{f(j)}^me),d(T_{f(j)}^lx_{i+1},T_{f(j)}^mx_{i+1}),d(T_{f(j)}^le,T_{f(j)}^me)\}.$$
Consider 
\begin{align*}
d(T_{f(j)}^lx_{i+1},T_{f(j)}^me) & \leq s d(T_{f(j)}^lx_{i+1},T_{f(j)}^l T_{f(i)}x_i)+s(T_{f(j)}^lT_{f(i)}x_i,T_{f(j)}^me)\\
& \leq s \thinspace \mbox{diam } \mathcal{O}_{T_{f(j)}}(x_{i+1},T_{f(i)}x_i)+ s \thinspace \mbox{diam } \mathcal{O}_{T_{f(j)}}(T_{f(i)}x_i,e)
\end{align*}
Since $e \in E$, by property $(\mathcal{P}_3)$ there exists $\bar{e} \in E$ such that $T_{f(i)}\bar{e}=e$. Using (\ref{equation20}) and (\ref{equation24}) we have
\begin{align*}
d(T_{f(j)}^lx_{i+1},T_{f(j)}^me)& \leq s\delta+ s \thinspace \mbox{diam } \mathcal{O}_{T_{f(j)}}(x_i, \bar{e}).
\end{align*}
Also, using (\ref{equation24}) 
\begin{align*}
d(T_{f(j)}^lx_{i+1},T_{f(j)}^mx_{i+1}) & \leq \mbox{diam } \mathcal{O}_{T_{f(j)}}(x_{i+1}) \leq \mbox{diam } \mathcal{O}_{T_{f(j)}}(x_{i+1},T_{f(i)}x_i) \leq \delta
\end{align*}
and by (\ref{equation20})
\begin{align*}
d(T_{f(j)}^le,T_{f(j)}^me)& \leq \mbox{diam } \mathcal{O}_{T_{f(j)}}(e) \leq \mbox{diam } \mathcal{O}_{T_{f(j)}}(T_{f(i)}x_i,e) \leq \mbox{diam } \mathcal{O}_{T_{f(j)}}(x_i,\bar{e}).
\end{align*}
Therefore, $\mbox{diam }\mathcal{O}_{T_{f(j)}}(x_{i+1},e) \leq s \delta+s \thinspace \mbox{diam } \mathcal{O}_{T_{f(j)}}(x_i, \bar{e})$. This implies that for all $i\in \{0,1, \ldots,\hat{n}\}$
\begin{align*}
\mbox{diam } \mathcal{O}_{T_{f(j)}}(x_i,e) & \leq s \delta+ s \thinspace \mbox{diam } \mathcal{O}_{T_{f(j)}}(x_{i-1},\bar{e})\\
&\leq (s+s^2)\delta+s^2 \thinspace \mbox{diam } \mathcal{O}_{T_{f(j)}}(x_{i-2},\breve{e})\\
& \thinspace \quad \vdots\\
& \leq  (s+s^2+\ldots+s^i)\delta+s^i \thinspace \mbox{diam } \mathcal{O}_{T_{f(j)}}(x_0, \tilde{e})\\
& \leq  i s^i \delta+s^i \thinspace  \mbox{diam } \mathcal{O}_{T_{f(j)}}(x_0, \tilde{e})\\
& \leq \hat{n} s^{\hat{n}} \delta+ s^{\hat{n}} \thinspace \mbox{diam } \mathcal{O}_{T_{f(j)}}(x_0, \tilde{e}).
\end{align*}
Using (\ref{equation26}) and (\ref{equation27}) we deduce that
$$\mbox{diam } \mathcal{O}_{T_{f(j)}}(x_i,e) \leq s^{\hat{n}}(2K-1)$$
which implies that $\{x_i\}_{i=1}^{\hat{n}} \subset Z$. This proves that property $(\mathcal{P}_6)$ is satisfied. Suppose that $f \in \mathfrak{F}$ and the sequence $\{x_i\}_{i=0}^{\infty} \subset X$ satisfies (\ref{equation23}) and (\ref{equation24}). Consider
$$\mbox{diam } \mathcal{O}_{T_i}(x_0,e)= \sup\limits_{l,m \in \mathbb{N}} \{d(T_i^lx_0,T_i^me),d(T_i^l x_0,T_i^m x_0),d(T_i^le,T_i^me)\}.$$
By (\ref{equation23}) and (\ref{equation25}) we have 
\begin{align*}
d(T_i^lx_0,T_i^me)& \leq s d(T_i^lx_0,T_i^l \theta)+sd(T_i^l \theta, T_i^m e)\\
& < s \thinspace \mbox{diam } \mathcal{O}_{T_i}(x_0,\theta)+s \thinspace \mbox{diam } \mathcal{O}_{T_i}(\theta,e)\\
&\leq 2K-2.
\end{align*}
Also, using (\ref{equation23})
\begin{align*}
d(T_i^lx_0,T_i^mx_0)& \leq \mbox{diam } \mathcal{O}_{T_i}(x_0)\leq \mbox{diam } \mathcal{O}_{T_i}(x_0,\theta) \leq \frac{K}{s}
\end{align*}
and using (\ref{equation25})
\begin{align*}
d(T_i^le,T_i^me)& \leq \mbox{diam } \mathcal{O}_{T_i}(e) \leq \mbox{diam } \mathcal{O}_{T_i}(\theta,e) \leq \frac{K-2}{s}.
\end{align*}
Therefore, $\mbox{diam } \mathcal{O}_{T_i}(x_0,e) \leq 2K-2$ which implies that (\ref{equation27}) holds for all $i \in \mathbb{N}_0$ and $e \in E$. Let $p \in \mathbb{N}_0$ and
$$\mbox{diam } \mathcal{O}_{T_i}(x_p,e) \leq 2K-2 \thinspace \thinspace \mbox{for all } \thinspace i \in \mathbb{N}_0 \thinspace \thinspace \mbox{and } \thinspace e \in E.$$
Then by property $(\mathcal{P}_6)$ we infer that $\{x_i\}_{i=p}^{p+\hat{n}} \subset Z$ and $\mbox{diam } \mathcal{O}_{T_i}(x_{p+\hat{n}},e) \leq s^{\hat{n}}(2K-1)$. This is true for all $p \in \mathbb{N}_0$. Therefore, 
$$\mbox{diam } \mathcal{O}_{T_i}(x_j,e) \leq s^k(2K-1)$$
for some $k \in \mathbb{N}$, and for all $i,j \in \mathbb{N}_0$ and $e \in E$ which implies that $\{x_i\}_{i=0}^{\infty} \subset Z$. Together with (\ref{equation24}) and property $(\mathcal{P}_5)$ we conclude that $d(x_i,E)< \epsilon$ for all $i \geq \hat{n}$.
\end{proof}
\begin{theorem}\label{theorem4}
Let $(X,d)$ be a complete $b$-metric space with coefficient $s \geq 1$. Let $E$ be a nonempty, bounded and closed subset of $X$. For each $i \in  \mathbb{N}_0$, let $T_i:X \rightarrow X$ be a weak quasi-contraction mapping, $\mathcal{O}$-$b$ continuous and satisfy  and $(\ref{equation20})$. Let $\mathfrak{F}$ be a nonempty set of mappings $f: \mathbb{N}_0 \rightarrow \mathbb{N}_0$ with  property $(\mathcal{F})$. Suppose  that properties $(\mathcal{P}_3)$ and  $(\mathcal{P}_4)$ hold. 
Let $\{ \delta_i\}_{i=0}^{\infty}$ be a sequence of positive numbers such that $\lim\limits_{i \rightarrow \infty} \delta_i=0$.
Then for each $\epsilon,K>0$, there exist $\delta>0$ and $\tilde{n} \in \mathbb{N}$ such that for each $f \in \mathfrak{F}$ and each sequence $\{x_i \}_{i=0}^{\infty} \subset X$ satisfying
\begin{align}\label{equation33}
\mbox{diam } \mathcal{O}_{T_j}(x_i,\theta) \leq \frac{K}{s} \thinspace \thinspace \mbox{for all } \thinspace i,j \in \mathbb{N}_0
\end{align}
and
\begin{align}\label{equation34}
\mbox{diam } \mathcal{O}_{T_{f(j)}}(x_{i+1},T_{f(i)}x_i) \leq \delta_i \thinspace \thinspace \mbox{for all } \thinspace i,j \in \mathbb{N}_0,
\end{align}
the following inequality is satisfied
$$d(x_i,E)< \epsilon \thinspace \thinspace \mbox{for all } \thinspace i \geq \tilde{n}.$$
\end{theorem}
\begin{proof}
By Theorem \ref{theorem3} there exist $\delta>0$ and $\hat{n} \in \mathbb{N}$ such that the following property is satisfied:

$(\mathcal{P}_7)$  for each $f \in \mathfrak{F}$ and each sequence $\{x_i\}_{i=0}^{\infty} \subset X$ satisfying (\ref{equation23}) and (\ref{equation24}), 
the following inequality is satisfied
$$d(x_i,E)<\epsilon \thinspace \thinspace \mbox{for all } \thinspace i \geq \hat{n}.$$
Since $\lim\limits_{i \rightarrow \infty} \delta_i=0$, there exists $n_1 \in \mathbb{N}$ such that 
\begin{align}\label{equation35}
\delta_i< \delta \thinspace \thinspace \mbox{for all } \thinspace i \geq n_1.
\end{align}
Let $n'= \max \{\hat{n},n_1\}$ and $\tilde{n} \geq 2 \hat{n}+n_1$. Let $n \geq \tilde{n}$ be an integer. It suffices to prove that $d(x_n,E)<\epsilon$. Set $z_i=x_{i+n'}$ and $\hat{f}(i)=f(i+n')$ for all $i \in \mathbb{N}_0$. Since $n' \in \mathbb{N}$, by property $(\mathcal{F})$ we deduce that $\hat{f} \in \mathfrak{F}$. Using (\ref{equation34}) and (\ref{equation35}) 
\begin{align*}
\mbox{diam } \mathcal{O}_{T_{\hat{f}(j)}}(z_{i+1},T_{\hat{f}(i)}z_i)&=\mbox{diam } \mathcal{O}_{T_{f(j+n')}}(x_{i+n'+1},T_{f(i+n')}x_{i+n'})\\
&\leq \delta_{i+n'}\\
& < \delta.
\end{align*}
Also, by (\ref{equation33})
$$\mbox{diam } \mathcal{O}_{T_i}(z_0, \theta)= \mbox{diam } \mathcal{O}_{T_i}(x_{n'},\theta) \leq \frac{K}{s}.$$
Therefore, by property $(\mathcal{P}_7)$ we deduce that
\begin{align}\label{equation36}
d(z_i,E) < \epsilon \thinspace \thinspace \mbox{for all } i \geq \hat{n}.
\end{align}
As $n \geq \tilde{n}$ and $n'=\max \{\hat{n},n_1\}$, $n-n' \geq \tilde{n}-n' \geq \hat{n}$. Therefore, by (\ref{equation36}) we conclude that
$$d(x_n,E)=d(z_{n-n'},E) < \epsilon.$$
\end{proof}
\section{Weak Ergodic Theorems}
Pustylnik et al. \cite{11} and Butnariu et al. \cite{4} obtained weak ergodic theorems in metric spaces. Motivated by them, we prove weak ergodic theorems in the setting of $b$-metric spaces. 
\begin{theorem}\label{theorem5}
Let $(X,d)$ be a complete $b$-metric space with coefficient $s \geq 1$. For each $i \in  \mathbb{N}_0$, let $T_i:X \rightarrow X$ be a weak quasi-contraction mapping satisfying $(\ref{equation20})$. Let $\mathfrak{F}$ be a set of mappings $f: \mathbb{N}_0 \rightarrow \mathbb{N}_0$ with property $(\mathcal{F})$. Suppose that the following property holds:

$(\mathcal{P}_8)$ for each $\epsilon>0$, there exists a natural number $n_{\epsilon}$ such that for each $f \in \mathfrak{F}$ and each $x,y \in X$
\begin{align*}
d(T_{f(n_{\epsilon})}T_{f(n_{\epsilon}-1)}\ldots T_{f(1)}T_{f(0)}x,T_{f(n_{\epsilon})}T_{f(n_{\epsilon}-1)}\ldots T_{f(1)}T_{f(0)}y) < \epsilon.
\end{align*}
Then for each $\epsilon>0$, there exist $\delta>0$ and $n \in \mathbb{N}$ such that for each $f \in \mathfrak{F}$ and each pair of sequences $\{x_i\}_{i=0}^{\infty}, \{y_i\}_{i=0}^{\infty} \subset X$ satisfying
\begin{align}
\mbox{diam }\mathcal{O}_{T_{f(j)}}(x_{i+1},T_{f(i)}x_i) \leq \delta \thinspace \thinspace \mbox{for all } \thinspace i,j \in \mathbb{N}_0, \label{eq38}\\
\mbox{diam }\mathcal{O}_{T_{f(j)}}(y_{i+1},T_{f(i)}y_i) \leq \delta \thinspace \thinspace \mbox{for all }  \thinspace i,j \in \mathbb{N}_0, \label{eq39}
\end{align}
the following inequality is satisfied
$$d(x_i,y_i) < \epsilon \thinspace \thinspace \mbox{for all } i \geq n.$$
\end{theorem}
\begin{proof}
Suppose that $\epsilon>0$ is given. Then by property $(\mathcal{P}_8)$, there exists $n \in \mathbb{N}$ such that for each $f \in \mathfrak{F}$ and each $x,y \in X$ we have
\begin{align}\label{eq40}
d(T_{f(n-2)}T_{f(n-3)}\ldots T_{f(1)}T_{f(0)}x,T_{f(n-2)}T_{f(n-3)}\ldots T_{f(1)}T_{f(0)}y) < \frac{\epsilon}{2s^2}. 
\end{align}
Choose  a real number $\delta$ such that 
\begin{align}\label{eq41}
0<\delta < \frac{\epsilon}{4ns^{n+1}}.
\end{align}
Suppose that $f \in \mathfrak{F}$ and the sequences $\{x_i\}_{i=0}^{\infty}, \{y_i\}_{i=1}^{\infty} \subset X$ satisfy (\ref{eq38}) and (\ref{eq39}), respectively. Let $m \geq n$ be an integer. It suffices to show that $d(x_m,y_m)< \epsilon$. Set $\bar{f}(i)=f(i+m-n+1)$, $\bar{x}_i=x_{i+m-n+1}$ and $\bar{y}_i=y_{i+m-n+1}$ for all $i \in \mathbb{N}_0$. Since $m \geq n$, using property $(\mathcal{F})$ we get $\bar{f} \in \mathfrak{F}$. Using (\ref{eq38}) 
\begin{align}\label{eqq1}
d(\bar{x}_{i+1},T_{\bar{f}(i)}\bar{x}_i)= d(x_{i+m-n+2},T_{f(i+m-n+1)}x_{i+m-n+1}) \leq \delta.
\end{align}
Similarly, using (\ref{eq39})
\begin{align}\label{eqq2}
d(\bar{y}_{i+1},T_{\bar{f}(i)} \bar{y}_i)=d(y_{i+m-n+2},T_{f(i+m-n+1)}x_{i+m-n+1}) \leq \delta.
\end{align}
Set $\tilde{x}_0=\bar{x}_0$ and $\tilde{x}_{i+1}=T_{\bar{f}(i)}\tilde{x}_i$ for all $i \in \mathbb{N}_0$. Also, set $\tilde{y}_0=\bar{y}_0$ and $\tilde{y}_{i+1}=T_{\bar{f}(i)}\tilde{y}_i$ for all $i \in \mathbb{N}_0$. In view of (\ref{eq40}),
\begin{align}\label{eqq5}
d(\tilde{x}_{n-1},\tilde{y}_{n-1})< \frac{\epsilon}{2s^2}.
\end{align}
Proceeding as in the proof of Theorem \ref{theorem1} using (\ref{eq38}) and (\ref{eqq1}) we get
\begin{align}\label{eq42}
\mbox{diam } \mathcal{O}_{T_{\bar{f}(j)}}(\tilde{x}_i,\bar{x}_i) \leq (i+1)s^i\delta \thinspace \thinspace \mbox{for all } \thinspace i,j \in \mathbb{N}_0.
\end{align} 
Similarly, using (\ref{eq39}) and (\ref{eqq2}) we have  
\begin{align}\label{eqq3}
\mbox{diam } \mathcal{O}_{T_{\bar{f}(j)}}(\tilde{y}_i, \bar{y}_i) \leq (i+1)s^i \delta \thinspace \thinspace \mbox{for all } \thinspace i,j \in \mathbb{N}_0.
\end{align}
Using (\ref{eq41}), (\ref{eqq5}),  (\ref{eq42})  and (\ref{eqq3})  we deduce that
\begin{align*}
d(x_m,y_m)&=d(\bar{x}_{n-1},\bar{y}_{n-1})\\
& \leq s d(\bar{x}_{n-1},\tilde{x}_{n-1})+s^2 d(\tilde{x}_{n-1},\tilde{y}_{n-1})+s^2 d(\tilde{y}_{n-1}, \bar{y}_{n-1})\\
& \leq ns^n \delta+s^2 \frac{\epsilon}{2s^2}+ ns^{n+1}\delta\\
& < \frac{\epsilon}{4}+\frac{\epsilon}{2}+\frac{\epsilon}{4}\\
&= \epsilon.
\end{align*}
\end{proof}
\begin{theorem}\label{theorem6}
Let $(X,d)$ be a complete $b$-metric space with coefficient $s \geq 1$. For each $i \in  \mathbb{N}_0$, let $T_i:X \rightarrow X$ be a weak quasi-contraction mapping satisfying $(\ref{equation20})$. Let $\mathfrak{F}$ be a nonempty set of mappings $f: \mathbb{N}_0 \rightarrow \mathbb{N}_0$ with property
$(\mathcal{F})$. Suppose that property $(\mathcal{P}_{8})$ holds.  Let $\{\delta_i \}_{i=0}^{\infty}$ be a sequence of positive numbers such that $\lim\limits_{i \rightarrow \infty}\delta_i=0$. Then for each $\epsilon>0$, there exists $\check{n} \in \mathbb{N}$ such that for each $f \in \mathfrak{F}$ and each pair of sequences $\{x_i\}_{i=0}^{\infty}, \{y_i\}_{i=0}^{\infty} \subset X$ satisfying
 \begin{align}
 \mbox{diam } \mathcal{O}_{T_{f(j)}}(x_{i+1},T_{f(i)}x_i) \leq \delta_i \thinspace \thinspace \mbox{for all } \thinspace i,j \in \mathbb{N}_0, \label{equation38}\\
  \mbox{diam } \mathcal{O}_{T_{f(j)}}(y_{i+1},T_{f(i)}y_i) \leq \delta_i \thinspace \thinspace \mbox{for all } \thinspace i,j \in \mathbb{N}_0, \label{equation39}
  \end{align}
  the following inequality is satisfied
  $$d(x_i,y_i)< \epsilon \thinspace \thinspace \mbox{for all } \thinspace i \geq \check{n}.$$
\end{theorem}
\begin{proof}
Suppose that $\epsilon>0$ is given. Then by Theorem \ref{theorem5} there exist $\delta>0$ and $n \in \mathbb{N}$ such that the following property is satisfied:

$(\mathcal{P}_9)$ for each $f \in \mathfrak{F}$ and each pair of sequences $\{x_i\}_{i=0}^{\infty}, \{y_i\}_{i=0}^{\infty} \subset X$ satisfying
$$\mbox{diam } \mathcal{O}_{T_{f(j)}}(x_{i+1},T_{f(i)}x_i) \leq \delta \thinspace \thinspace \mbox{for all } i,j \in \mathbb{N}_0,$$
$$\mbox{diam } \mathcal{O}_{T_{f(j)}}(y_{i+1},T_{f(i)}x_i) \leq \delta \thinspace \thinspace \mbox{for all } i,j \in \mathbb{N}_0,$$
the following inequality holds
$$d(x_i,y_i) < \epsilon \thinspace \thinspace \mbox{for all } \thinspace i \geq n.$$
Since $\lim\limits_{i \rightarrow \infty} \delta_i=0$, there exists $n_1 \in \mathbb{N}$ such that
\begin{align}\label{equation40}
\delta_i< \delta \thinspace \thinspace \mbox{for all } i \geq n_1.
\end{align}
Let $n'=\max \{n,n_1\}$ and $\check{n} \geq 2n+n_1$.  Suppose that $f \in \mathfrak{F}$ and the sequences $\{x_i\}_{i=0}^{\infty}, \{y_i\}_{i=0}^{\infty} \subset X$ satisfy  (\ref{equation38}) and (\ref{equation39}), respectively. Let $m \geq \check{n}$ be an integer. It suffices to show that $d(x_m,y_m) < \epsilon$. 
Set $\check{f}(i)=f(i+n')$, $\check{x}_i=x_{i+n'}$ and $\check{y}_i=y_{i+n'}$ for all $i \in \mathbb{N}_0$. Since $n' \in \mathbb{N}$, by property $(\mathcal{F})$ $\check{f} \in \mathfrak{F}$. Using (\ref{equation38}) and (\ref{equation40})
\begin{align*}
\mbox{diam } \mathcal{O}_{T_{\check{f}(j)}} (\check{x}_{i+1}, T_{\check{f}(i)} \check{x}_i)&= \mbox{diam } \mathcal{O}_{T_{f(j+n')}}(x_{i+n'+1},T_{{f}(i+n')}x_{i+n'})\\
& \leq \delta_{i+n'}\\
& < \delta.
\end{align*}
Similarly, using (\ref{equation39}) and (\ref{equation40}) $\mbox{diam } \mathcal{O}_{T_{\bar{f}(j)}}(\check{y}_{i+1},T_{\check{f}(i)}\check{y}_i) \leq  \delta$. Therefore, by property $(\mathcal{P}_9)$ we deduce that $d(\check{x}_i,\check{y}_i) < \epsilon$ for all $i \geq n$. Since $m\geq \check{n}$, $m-n' \geq n$. This gives  $d(x_m,y_m)=d(\check{x}_{m-n'},\check{y}_{m-n'})< \epsilon$. 
\end{proof}

\section*{Acknowledgements}
 The $^*$corresponding author is supported by University Grants Commission Research Grant (Ref. No. NFO-$2018$-$19$-OBC-HAR$72140$).
 
\textbf{Anuradha Gupta}\\
 Department of Mathematics, Delhi College of Arts and Commerce,\\
  University of Delhi, Netaji Nagar, \\
  New Delhi-110023, India.\\
  \vspace{0.2cm}
 email: dishna2@yahoo.in\\
   \textbf{Manu Rohilla}\\
  Department of Mathematics, University of Delhi, \\
  New Delhi-110007, India.\\
  email: manurohilla25994@gmail.com

\begin{thebibliography}{99}
  \bibitem{1} M. Bessenyei, The contraction principle in extended context, Publ. Math. Debrecen {\bf 89} (2016), no.~3, 287--295.
  
  \bibitem{2} M. Boriceanu, M. Bota\ and\ A. Petru\c{s}el, Multivalued fractals in $b$-metric spaces, Cent. Eur. J. Math. {\bf 8} (2010), no.~2, 367--377.
  
  \bibitem{3} D. Butnariu, S. Reich\ and\ A. J. Zaslavski, Asymptotic behavior of inexact orbits for a class of operators in complete metric spaces, J. Appl. Anal. {\bf 13} (2007), no.~1, 1--11. 
  
 \bibitem{4} D. Butnariu, S. Reich\ and\ A. J. Zaslavski, Stable convergence theorems for infinite products and powers of nonexpansive mappings, Numer. Funct. Anal. Optim. {\bf 29} (2008), no.~3-4, 304--323.
 
  \bibitem{5} J. E. Cohen, Ergodic theorems in demography, Bull. Amer. Math. Soc. (N.S.) {\bf 1} (1979), no.~2, 275--295.
  
  \bibitem{6} S. Czerwik, Contraction mappings in $b$-metric spaces, Acta Math. Inform. Univ. Ostraviensis {\bf 1} (1993), 5--11.
   
 \bibitem{7} Z. D. Mitrovi\'{c}\ and\ N. Hussain, On weak quasicontractions in $b$-metric spaces, Publ. Math. Debrecen {\bf 94} (2019), no.~3-4, 289--298.
 
 \bibitem{8} R. D. Nussbaum, Some nonlinear weak ergodic theorems, SIAM J. Math. Anal. {\bf 21} (1990), no.~2, 436--460.
 
  \bibitem{9} E. Pustylnik, S. Reich\ and\ A. J. Zaslavski, Convergence to compact sets of inexact orbits of nonexpansive mappings in Banach and metric spaces, Fixed Point Theory Appl. {\bf 2008}, Art. ID 528614, 10 pp.
 
  \bibitem{10} E. Pustylnik, S. Reich\ and\ A. J. Zaslavski, Inexact orbits of nonexpansive mappings, Taiwanese J. Math. {\bf 12} (2008), no.~6, 1511--1523.
 
 \bibitem{11} E. Pustylnik, S. Reich\ and\ A. J. Zaslavski, Inexact infinite products of nonexpansive mappings, Numer. Funct. Anal. Optim. {\bf 30} (2009), no.~5-6, 632--645.
 
  \bibitem{12} S. Reich\ and\ A. J. Zaslavski, Convergence of generic infinite products of nonexpansive and uniformly continuous operators, Nonlinear Anal. {\bf 36} (1999), no.~8, Ser. A: Theory Methods, 1049--1065.
 
 \bibitem{13} S. Reich\ and\ A. J. Zaslavski, Asymptotic behavior of inexact infinite products of nonexpansive mappings in metric spaces, Z. Anal. Anwend. {\bf 33} (2014), no.~1, 101--117.
 
 \end{thebibliography}
\end{document}